\documentclass[12pt,oneside,reqno]{amsart}
\usepackage[all]{xy}
\usepackage{amsfonts,amsmath,oldgerm,amssymb,amscd}
\UseComputerModernTips
\numberwithin{equation}{section}
\usepackage[breaklinks]{hyperref}
\allowdisplaybreaks

\usepackage{enumerate}

\usepackage{caption} 
\newtheorem{theorem}{Theorem}
\newtheorem{proposition}[theorem]{Proposition}

\newtheorem{lemma}[subsection]{{\bf Lemma}}
\newtheorem{coro}[subsection]{{\bf Corollary}}

\newcommand{\C}{\mbox{$\mathbb C$}}     

\begin{document}

\title[Multiple Lucas Dirichlet series ]{Multiple Lucas Dirichlet series associated to additive and Dirichlet characters} 

\author[Meher]{N. K. Meher}
\address{Nabin Kumar Meher,  National Institute of Science Education and Research, Bhubaneswar , HBNI\\ P.O. Jatni, Khurda, Odisha -752050, India.}
\email{mehernabin@gmail.com}

\author[Rout]{S. S. Rout}
\address{Sudhansu Sekhar Rout, Institute of Mathematics and Applications\\ Andharua, Bhubaneswar, Odisha - 751029\\ India.}
\email{lbs.sudhansu@gmail.com; sudhansu@iomaorissa.ac.in}

\maketitle
\pagenumbering{arabic}
\pagestyle{headings}

\begin{abstract}
In this article, we obtain the analytic continuation of the multiple shifted Lucas zeta function, multiple Lucas $L$-function associated to Dirichlet characters and additive characters. We then compute a complete list of exact singularities and residues of these functions at these poles. Further, we show the rationality of the multiple Lucas $L$-functions associated with quadratic characters at negative integer arguments.
\end{abstract}

\section{Introduction}
Let $P$ and $Q$ be real number. Then the Lucas sequence of the first kind is defined by
\begin{equation}\label{eq1}
U_n = U_n (P, Q):= \frac{\alpha^n - \beta^n}{\alpha- \beta},\;\;  \mbox{for}\;\; n \geq 0
\end{equation}
 where $\alpha$ and $\beta$ are the roots of the quadratic equation $x^2 - Px + Q = 0$. If we choose
$P = 1$ and $Q = -1,$ then we obtain  the Fibonacci sequence which is  defined by the binary recurrence  
\[F_0 = 0, \;\; F_1 = 1,\;\;\;F_n = F_{n-1} + F_{n-2} , \quad (n \geq 2).\]
 
 Let $\chi$ be a Dirichlet character (mod $q$) and 
 \begin{equation}\label{eq2}
 \mathcal{L}_{U}(s, \chi):= \sum_{n =1}^{\infty}\frac{\chi(n)}{U_n^s}, \quad (\mbox{Re}(s) > 0)
 \end{equation}
 which is an analogue of the Dirichlet $L$-function $L(s, \chi) = \sum_{n =1}^{\infty}\chi(n)/n^s$. The function \eqref{eq2} can be analytically continued to the whole complex plane (see \cite{Kamano2013, Ram2001, Navas2001}). Further, it has been proved that if the character $\chi$ is non-principal, then $\mathcal{L}_{U}(s, \chi)$ is holomorphic on the real axis. The Fibonacci zeta function 
\begin{equation}{\label{eq3}}
\zeta_F(s) = \sum_{n =1}^{\infty} \frac{1}{F_n^{s}},  \quad (\mbox{Re}(s) > 1)
\end{equation}
is the special case of \eqref{eq2}. For the  special values of Fibonacci zeta function one can refer \cite{Jeannin1989, Duverney1997, Elsner2007, Ram2001, Navas2001}. 
\smallskip

In 1882, Hurwitz \cite{Hurwitz}, defined the ``shifted" Riemann zeta function, $\zeta(s; x)$ by the series 
\begin{equation}\label{eq4}
\zeta(s; x):= \sum_{n =0}^{\infty}\frac{1}{(n +x)^s}, \quad (\mbox{Re}(s) > 0)
\end{equation}
for any $x$ satisfying $0< x\leq 1$.  The analytic continuation for the Hurwitz zeta function gives us the analytic continuation for the Dirichlet $L$-function since 
\begin{equation}\label{eq5}
L(s, \chi):= \sum_{n =1}^{\infty}\frac{\chi(n)}{n^s} = q^{-s} \sum_{a\ (\mathrm{mod}\ q)} \chi (a) \zeta(s; a/q)
\end{equation}
 where $\chi$ is the Dirichlet character (mod $q$).
 
The special values of the Riemann zeta function stimulates  the study of the analytic continuation of the multiple zeta functions defined as: 
  \begin{equation}\label{eq6}
\zeta(s_1, \ldots,s_d):= \sum_{0<n_1<n_2\cdots<n_d}\frac{1}{{n_1}^{s_1} \cdots {n_d}^{s_d}}.
\end{equation}
There are several nice identities in the literature of multi zeta values. In recent years, several mathematicians are focusing their attention to the analytic properties of multiple zeta function. The series \eqref{eq6} converges absolutely for Re$(s_i) \geq 1$ where $2 \leq i \leq d$ and Re$(s_1) > 1$ and therefore, one can consider the problem of analytic continuation of \eqref{eq6} and this question has been studied by several authors (see for example \cite{Akiyama2001, Arakawa1999, Atkinson1949, Zhao2000}). Later, several generalizations of the multiple zeta functions were introduced and their analytic properties were discussed.  One of the important generalization is the multiple series associated to Dirichlet characters which we call as multiple Dirichlet $L$- functions. The multiple Dirichlet $L$-function is defined as 
\begin{equation}\label{eq7}
L(s_1, \ldots, s_d,\mid \chi_1, \ldots,\chi_d):= \sum_{0<n_1<n_2\cdots<n_d}\frac{\chi_1(n_1) \cdots \chi_{d}(n_d)}{n_{1}^{s_1}\cdots n_{d}^{s_d}}
\end{equation}
 where $\chi_1, \ldots, \chi_d$ are Dirichlet characters of same modulus $q\geq 2$. The analytic theory of the multiple Dirichlet $L$-functions has been studied by Akiyama and Ishikawa (\cite{Akiyama2002}, \cite{Mehta2017}).

 Recently, Meher and Rout  \cite{Meher2018} studied the meromorphic continuation of the multiple Lucas zeta function,
 \begin{equation}\label{eq8}
\zeta_U(s_1, \ldots,s_d):= \sum_{0<n_1<n_2\cdots<n_d}\frac{1}{U_{n_1}^{s_1} \cdots U_{n_d}^{s_d}},
\end{equation}
 where $(U_n)$ is the Lucas sequence of the first kind. Here the sum $s_1 + \cdots + s_d$ is called the weight of $\zeta_U(s_1, \ldots,s_d)$ and $d$ is called its depth. In this paper, we obtain the meromorphic continuation of multiple shifted Lucas zeta function and also multiple Lucas $L$-function associated with Dirichlet characters and additive characters.
 
 Let $\chi_1, \ldots, \chi_d$ be the Dirichlet characters of same modulus $q\geq 2$ and  $\chi_0$ be the principal character. Then {\em multiple shifted Lucas zeta function} and {\em multiple Lucas $L$-function} are defined respectively as
 \begin{equation}\label{eq9}
 \zeta_{U}^{d}(s_1, \ldots,s_d\mid r_1, \ldots, r_d):= \sum_{n_1 = 0,\cdots,  n_d = 0}^{\infty}  \frac{1}{U_{qn_1+r_1}^{s_1}} \cdots \frac{1}{U_{q(n_1 + \cdots + n_d)+(r_1+\cdots+r_d)}^{s_d}} ,
 \end{equation}
 and
  \begin{equation}\label{eq10}
 \mathcal{L}_{U}^{d}(s_1, \ldots, s_d\mid \chi_1, \ldots, \chi_d):= \sum_{0<n_1<n_2\cdots<n_d}\frac{\chi_1(n_1)}{U_{n_{1}}^{s_1}} \cdots \frac{\chi_d(n_d)}{U_{n_{d}}^{s_d}},
 \end{equation}
 where $q\in \mathbb{N}_{\geq 2}$ and $n_i, r_i\in \mathbb{N}$ for $1\leq i \leq d$. For convenience, we use $\zeta_{U}^{d}(\bf s\mid \bf r)$ and $\mathcal{L}_{U}^{d}(\bf s\mid \bf \chi)$ for $\zeta_{U}^{d}(s_1, \ldots,s_d\mid r_1, \ldots, r_d)$ and $ \mathcal{L}_{U}^{d}(s_1, \ldots, s_d\mid \chi_1, \ldots, \chi_d)$ respectively. Linear independence and irrationality results of the series \eqref{eq10} have been studied in \cite{E2017} when $d = 1 = s_1$ and $U_n = F_n$.

Further, we consider additive characters, i.e., group homomorphisms \[f_1, \ldots, f_d: \mathbb{Z}\longrightarrow \mathbb{C}^{*}\]and study the multiple Lucas Dirichlet series associated to these additive characters:
 \begin{align}\label{eq11}
 \begin{split}
 \mathcal{L}_{U}^{d}(s_1, \ldots, s_d\mid f_1, \ldots, f_d) &:= \sum_{0<n_1<n_2\cdots<n_d}\frac{f_1(n_1)}{U_{n_{1}}^{s_1}} \cdots \frac{f_d(n_d)}{U_{n_{d}}^{s_d}}\\
 &= \sum_{0<n_1<n_2\cdots<n_d}\frac{f_1(1)^{n_1}}{U_{n_{1}}^{s_1}} \cdots \frac{f_d(1)^{n_d}}{U_{n_{d}}^{s_d}}.
 \end{split}
 \end{align}
 In this case, we call this multiple Lucas series as {\em multiple Lucas additive $L$-function of depth $d$ associated to additive characters $f_1, \ldots, f_d$} which we simply write $\mathcal{L}_{U}^{d}(\bf s\mid \bf f)$. A prototypical example of such type of multiple series is:
 \[\sum_{n=1}^{\infty}\frac{(-1)^{n}}{F_n^s},\] 
 and irrationality of such series is obtained by Andr\'{e}-Jeannin \cite{Jeannin1989} when $s=1$.
 
The organization of the paper is as follows. In Section \ref{prem}, we prove that $\zeta_{U}^{d}(\bf s\mid \bf r)$ and $\mathcal{L}_{U}^{d}(\bf s\mid \bf \chi)$ series are absolutely convergent in the domain $\{(s_1, \ldots, s_d)\in \C^d | \; \sum_{i = j}^{d}\mathrm{Re}(s_i)  > 0, \; 1 \leq j \leq d \}$. We show that \eqref{eq11} is absolutely convergent in the domain $\{(s_1, \ldots, s_d)\in \C^d | \; \sum_{i = j}^{d}\mathrm{Re}(s_i)  > 0, \; 1 \leq j \leq d \}$ with assumption that the partial products $\prod_{k=i}^{d}\left|f_{k}(1)\right|\leq 1$ for all $1\leq i\leq d$. In Section \ref{sec3}, first we obtain the meromorphic  continuation of  $\zeta_{U}^{d}(\bf s\mid \bf r)$ and $\mathcal{L}_{U}^{d}(\bf s\mid \bf f)$. Also, we establish an explicit relation between multiple shifted Lucas zeta function $\zeta_{U}^{d}(\bf s\mid \bf r)$ and multiple Lucas $L$-function $\mathcal{L}_{U}^{d}(\bf s\mid \bf \chi)$. Using this relation, we prove the meromorphic continuation of $\mathcal{L}_{U}^{d}(\bf s\mid \bf \chi)$. Section \ref{sec4} is devoted to compute a complete list of poles and residues of the series $\mathcal{L}_{U}^{d}(\bf s\mid \bf \chi)$ and $\mathcal{L}_{U}^{d}(\bf s\mid \bf f)$ at these poles. Finally in Section \ref{sec5}, we show that the special values of the multiple Lucas zeta function associated to Dirichlet characters at negative integer arguments are rational.

 \section{Preliminaries}\label{prem}
Throughout this paper, we assume that
\begin{equation}\label{eq3b}
P > 0, Q \neq 0, D: = P^2 - 4 Q,\;\;\mbox{and}\;\; 
  \begin{cases} 
   Q < P -1 & \text{if } 0 < P \leq 2 \\
   Q \leq P-1       & \text{if } P>2.
  \end{cases}
  \end{equation}

For an integer $d \geq 1$, we consider the open subset of $\mathfrak{D}_d$ of $\mathbb{C}^d$:
\[\mathfrak{D}_d := \{(s_1, \ldots, s_d)\in \C^d | \; \sum_{i = j}^{d}\mathrm{Re}(s_i)  > 0, \; 1 \leq j \leq d \}.\]

We need the following proposition from \cite{Meher2018} to prove the absolute convergence of multiple shifted Lucas zeta function $\zeta_{U}^{d}(\bf s\mid \bf r)$ and multiple Lucas $L$-function $\mathcal{L}_{U}^{d}(\bf s\mid \bf \chi)$.
\begin{proposition}[\cite{Meher2018}]\label{prop1.1}
The series 
\[\sum_{0<n_1<\cdots < n_d}\frac{1}{U_{n_1}^{s_1}\cdots U_{n_d}^{s_d}}\] is absolutely convergent in the domain  $\mathfrak{D}_d$. 
\end{proposition}
\begin{proposition}\label{prop1}
The multiple shifted Lucas zeta function $\zeta_{U}^{d}(\bf s\mid \bf r) $
is absolutely convergent in $\mathfrak{D}_d$. 
\end{proposition}
\begin{proof}
From Proposition \ref{prop1.1}, we have 
\begin{equation}\label{eq3d}
\sum_{0<n_1<\cdots < n_d} \Big| \frac{1}{U_{n_1}^{s_1}\cdots U_{n_d}^{s_d}}\Big| \leq \sum_{n_1 = 1}^{\infty} \left| \frac{1}{U_{n_1}^{s_1}} \right|\sum_{n_2 = 1}^{\infty} \left| \frac{1}{U_{n_1 + n_2}^{s_2}} \right| \cdots \sum_{n_d = 1}^{\infty} \left|\frac{1}{U_{n_1 +\cdots + n_d}^{s_d}}\right| < \infty.
\end{equation}
Note that 
\begin{align}\label{eq3c}
\begin{split}
 \sum_{n_1 = 0,\cdots,  n_d = 0}^{\infty} & \Bigg| \frac{1}{U_{qn_1+r_1}^{s_1}} \cdots \frac{1}{U_{q(n_1 + \cdots + n_d)+(r_1+\cdots+r_d)}^{s_d}}\Bigg|\\
 \leq & \sum_{n_1 = 0,\cdots,  n_d = 0}^{\infty} \Bigg| \frac{1}{U_{qn_1+r_1}^{s_1}} \Bigg| \cdots \Bigg|\frac{1}{U_{q(n_1 + \cdots + n_d)+(r_1+\cdots+r_d)}^{s_d}}\Bigg|\\
\leq &\sum_{n_1 = 1}^{\infty} \left| \frac{1}{U_{n_1}^{s_1}} \right|\sum_{n_2 = 1}^{\infty} \left| \frac{1}{U_{n_1 + n_2}^{s_2}} \right| \cdots \sum_{n_d = 1}^{\infty} \left|\frac{1}{U_{n_1 +\cdots + n_d}^{s_d}}\right|  < \infty. 
\end{split}
\end{align}
Hence, the Proposition \ref{prop1}  follows from \eqref{eq3d} and \eqref{eq3c}.
\end{proof}

\begin{proposition}\label{prop2.1}
For a positive integer $d\geq 1,$ let $\chi_1, \ldots, \chi_d$ be the Dirichlet characters of same modulus $q\geq 2$. Then the infinite sum 
\[\sum_{0<n_1<n_2\cdots<n_d}\frac{\chi_1(n_1)}{U_{n_{1}}^{s_1}} \cdots \frac{\chi_d(n_d)}{U_{n_{d}}^{s_d}}\] is absolutely convergent in $\mathfrak{D}_d$. 
\end{proposition}
\begin{proof}
Observe that
\begin{align*}
&\sum_{0<n_1<n_2\cdots<n_d} \frac{\chi_1(n_1)}{U_{n_{1}}^{s_1}} \cdots \frac{\chi_d(n_d)}{U_{n_{d}}^{s_d}} =  \sum_{n_1 = 1}^{\infty}\frac{\chi_1(n_1)}{U_{n_1}^{s_1}}\sum_{n_2 = 1}^{\infty} \frac{\chi_2(n_1+n_2)}{U_{n_1 + n_2}^{s_2}} \cdots \sum_{n_d = 1}^{\infty} \frac{\chi_d(n_1 + \cdots+ n_d)}{U_{n_1 +\cdots + n_d}^{s_d}}.
\end{align*}
Since $|\chi(n)| \leq 1$ for any natural number $n$,  by Proposition \ref{prop1.1}, we obtain that
 the series $\mathcal{L}_{U}^d(\bf s\mid \bf \chi)$ converges absolutely in $\mathfrak{D}_d.$ 
\end{proof}
\begin{proposition}\label{prop3}
For a positive integer  $d\geq 1,$ let $f_1, \ldots, f_d$ be  additive characters such that the partial products $\prod_{k=i}^{d}\left|f_{k}(1)\right|\leq 1$ for all $1\leq i\leq d$. Then the infinite sum 
\[\sum_{0<n_1<n_2\cdots<n_d}\frac{f_1(n_1)}{U_{n_{1}}^{s_1}} \cdots \frac{f_d(n_d)}{U_{n_{d}}^{s_d}}\] converges absolutely in $\mathfrak{D}_d$. 
\end{proposition}
\begin{proof}
Note that
\begin{align*}
\sum_{0<n_1<n_2\cdots<n_d} & \frac{f_1(n_1)}{U_{n_{1}}^{s_1}} \cdots \frac{f_d(n_d)}{U_{n_{d}}^{s_d}} =  \sum_{n_1 = 1}^{\infty}\frac{f_1(n_1)}{U_{n_1}^{s_1}}\sum_{n_2 = 1}^{\infty} \frac{f_2(n_1+n_2)}{U_{n_1 + n_2}^{s_2}} \cdots \sum_{n_d = 1}^{\infty} \frac{f_d(n_1 + \cdots+ n_d)}{U_{n_1 +\cdots + n_d}^{s_d}}\\
& = \sum_{n_1 = 1}^{\infty}\frac{f_1(1)^{n_1}}{U_{n_1}^{s_1}}\sum_{n_2 = 1}^{\infty} \frac{f_2(1)^{n_1+n_2}}{U_{n_1 + n_2}^{s_2}} \cdots \sum_{n_d = 1}^{\infty} \frac{f_d(1)^{n_1+ \cdots +n_d}}{U_{n_1 +\cdots + n_d}^{s_d}}\\
& = \sum_{n_1 = 1}^{\infty}\frac{\prod_{k=1}^{d}f_{k}(1)^{n_1}}{U_{n_1}^{s_1}}\sum_{n_2 = 1}^{\infty} \frac{\prod_{k=2}^{d}f_{k}(1)^{n_2}}{U_{n_1 + n_2}^{s_2}} \cdots \sum_{n_d = 1}^{\infty} \frac{f_{d}(1)^{n_d}}{U_{n_1 +\cdots + n_d}^{s_d}}.
\end{align*}
Since $\prod_{k=i}^{d}\left|f_{k}(1)\right|\leq 1$ for all $1\leq i\leq d,$ by using Proposition \ref{prop1.1}, we have 
\begin{align*}
\sum_{0<n_1<n_2\cdots<n_d}  \left | \frac{f_1(n_1)}{U_{n_{1}}^{s_1}} \cdots \frac{f_d(n_d)}{U_{n_{d}}^{s_d}} \right | 
  \leq \sum_{n_1 = 1}^{\infty} \left| \frac{1}{U_{n_1}^{s_1}} \right|\sum_{n_2 = 1}^{\infty} \left| \frac{1}{U_{n_1 + n_2}^{s_2}} \right| \cdots \sum_{n_d = 1}^{\infty} \left|\frac{1}{U_{n_1 +\cdots + n_d}^{s_d}}\right | < \infty.
\end{align*}
This completes the proof of the Proposition \ref{prop3}.
\end{proof}

\section{Analytic continuation of multiple Lucas Dirichlet series associated to additive and Dirichlet characters}\label{sec3}
From here onward, we use the following notations: 
\[\mbox{Re}(s_i) = \sigma_i (1\leq i \leq d), \quad \ell_{Q, k_j, \ldots, k_d} = 
\begin{cases}
    0      & \quad \text{if } Q>0\\
    k_j + \cdots + k_d & \quad \text{if } Q <0.
  \end{cases}\]
\begin{theorem}\label{thm1}
  The multiple shifted Lucas zeta function $\zeta_{U}^{d}(\bf s\mid \bf r)$ 
of depth $d$ can be  meromorphically continued on all of $\mathbb{C}^d$ with exact list of poles on the hyper planes 
\[s_j + \cdots + s_d = -2(k_j+\cdots + k_d) + (k_j + \cdots + k_d) \frac{\log |Q|}{\log \alpha} + \left(\frac{2n}{q} + \ell_{Q, k_j, \ldots, k_d}  \right) \frac{\pi i }{\log \alpha} \]
for $1\leq j \leq d$ where $ k_j, n \in \mathbb{Z} $ with $ k_j \geq 0$. 
\end{theorem}
\begin{proof} Since $\alpha - \beta = \sqrt{D}$ and $ \alpha \beta = Q,$ for any $z \in \mathbb{C},$ we have
\begin{align}\label{eq3.1}
\begin{split}
U_{qn+r} ^{z} =\left(\frac{\alpha^{qn+r} - \beta^{qn+r}}{\alpha - \beta}\right)^z = 
D^{-z/2} \sum_{k = 0} ^{\infty}\binom{z}{k}(-1)^{k} Q^{(qn+r)k}\alpha^{(qn+r)(z-2k)}.
\end{split}
\end{align}
Since the series $\zeta_U^d(\bf s\mid \bf r)$ is absolutely convergent, by interchanging the summation,  we have 
\begin{align}\label{eq3.3}
\zeta_U^d(\bf s\mid \bf r) \nonumber
=  & D^{(s_1 +\cdots+ s_d)/2} \sum_{k_1 =0}^{\infty}\binom{-s_1}{k_1} (-1)^{k_1} \cdots  \sum_{k_d =0}^{\infty}\binom{-s_d}{k_d} (-1)^{k_d} \\\nonumber
&\times Q^{(k_1+ \cdots k_d)r_1}\left( \alpha^{-(s_1 + \cdots +s_d + 2(k_1 + \cdots +k_d))}\right)^{r_1}\cdots Q^{k_dr_d}\left(\alpha^{-(s_d + 2k_d)} \right)^{r_d}\\\nonumber
&\sum_{n_1, \ldots, n_d = 0}^{\infty} 
Q^{(k_1+ \cdots k_d)(qn_1)}\left( \alpha^{-q(s_1 + \cdots +s_d + 2(k_1 + \cdots +k_d))}\right)^{n_1} \cdots  Q^{k_dqn_d}\left(\alpha^{-q(s_d + 2k_d)} \right)^{n_d}\\ \nonumber
&= D^{(s_1 +\cdots+ s_d)/2} \sum_{k_1 =0}^{\infty}\binom{-s_1}{k_1} (-1)^{k_1} \cdots  \sum_{k_d =0}^{\infty}\binom{-s_d}{k_d} (-1)^{k_d} \\
&\times \frac{Q^{(k_1+ \cdots k_d)r_1} \alpha^{-r_1(s_1 + \cdots +s_d + 2(k_1 + \cdots +k_d))}}{1 -  Q^{q(k_1+\cdots + k_d)}\alpha^{-q(s_1 +\cdots + s_d + 2(k_1 + \cdots+k_d))}}  \cdots \frac{ Q^{k_dr_d}\alpha^{-r_d(s_d + 2k_d)}}{(1- Q^{qk_d}\alpha^{-q(s_d + 2k_d})} \\ \nonumber
&= D^{(s_1 +\cdots+ s_d)/2} \sum_{k_1 =0}^{\infty}\binom{-s_1}{k_1} (-1)^{k_1} \cdots  \sum_{k_d =0}^{\infty}\binom{-s_d}{k_d} (-1)^{k_d} \\\nonumber
&\times \frac{Q^{k_1r_1} \alpha^{-r_1(s_1 + 2k_1)}}{1 -  Q^{q(k_1+\cdots + k_d)}\alpha^{-q(s_1 +\cdots + s_d + 2(k_1 + \cdots+k_d))}}  \cdots \frac{ Q^{k_d(r_1+\cdots + r_d)}\alpha^{-(r_1+\cdots + r_d)(s_d + 2k_d)}}{(1- Q^{qk_d}\alpha^{-q(s_d + 2k_d})}. 
\end{align}
This infinite series is a holomorphic function on  $\mathbb{C}^d$ except for the poles derived from  the functions
\[ g_{k_j, \ldots, k_d}(s_j, \ldots, s_d) := \frac{Q^{r_j(k_j+\cdots + k_d)}\alpha^{-r_j(s_j + \cdots +s_d + 2(k_j + \cdots +k_d))}}{1 -  Q^{q(k_j+\cdots + k_d)}\alpha^{-q(s_j +\cdots + s_d + 2(k_j + \cdots+k_d))}},\]
 for $1\leq j \leq d$. Hence  $\zeta_U^d(\bf s\mid {\bf r})$ meromorphically continued on $\mathbb{C}^d$ with poles on the hyperplanes for
$1\leq j \leq d$
\begin{equation}\label{eq3.6}
s_j + \cdots + s_d = -2(k_j+\cdots + k_d) + (k_j + \cdots + k_d) \frac{\log |Q|}{\log \alpha} + \left(\frac{2n}{q} + \ell_{Q, k_j, \ldots, k_d}  \right) \frac{\pi i }{\log \alpha}
\end{equation}
where $ k_j, n \in \mathbb{Z} $ with $ k_j \geq 0$ and 
\[\ell_{Q, k_j, \ldots, k_d} = 
\begin{cases}
    0      & \quad \text{if } Q>0\\
    k_j + \cdots + k_d & \quad \text{if } Q <0.
  \end{cases}\] 
  In fact, these poles are simple.
 \end{proof} 
 
Observe that 
 \begin{align}\label{eq3.6d}
 \begin{split}
& \mathcal{L}_{U}^d({\bf s}\mid {\bf \chi}) =  \sum_{n_1 = 1}^{\infty}\frac{\chi_1(n_1)}{U_{n_1}^{s_1}}\sum_{n_2 = 1}^{\infty} \frac{\chi_2(n_1+n_2)}{U_{n_1 + n_2}^{s_2}} \cdots \sum_{n_d = 1}^{\infty} \frac{\chi_d(n_1 + \cdots+ n_d)}{U_{n_1 +\cdots + n_d}^{s_d}}\\
& = \sum_{r_1 = 1}^{q} \sum_{r_2 = 1}^{q}\cdots \sum_{r_d = 1}^{q} \sum_{n_1 =1}^{\infty}\frac{\chi_1(r_1)}{U_{qn_1 + r_1}^{s_1}}\cdots \sum_{n_d = 1}^{\infty} \frac{\chi_d(r_1+\cdots+ r_d)}{U_{q(n_1 + \cdots+n_d) + r_1 +\cdots r_d}^{s_d}} \\
& = \sum_{r_1 = 1}^{q}\sum_{r_2 = 1}^{q}\cdots \sum_{r_d=1}^{q} \chi_1(r_1) \chi_{2}(r_1 + r_2)\cdots \chi_d(r_1+\cdots + r_d) \zeta_{U}^{d}(\bf s\mid \bf r). 
\end{split}
\end{align}
Thus the function $\mathcal{L}_{U}^{d}(\bf s\mid \bf \chi)$ is a linear combination of the $\zeta_{U}^{d}(\bf s\mid \bf r)$. Therefore, we have the following result.
 
 \begin{theorem}\label{thm2}
For any positive integer $d\geq 1,$ let $\chi_1, \ldots, \chi_d$ be the Dirichlet characters of same modulus $q\geq 2$, then the multiple Lucas $L$-function $\mathcal{L}_{U}^{d}(\bf s\mid  {\bf \chi})$ 
of depth $d$ can be meromorphically continued on all of $\mathbb{C}^d$ with possible poles on the hyper planes given by \eqref{eq3.6}.
\end{theorem}

 Now, we discuss the analytic continuation of the multiple Lucas function associated with  additive characters 
$\mathcal{L}_{U}^{d}(\bf s\mid \bf f)$.
 \begin{theorem}\label{thm3}
For any positive integer $d\geq 1,$ let $f_1, \ldots, f_d$ be the  additive characters with $\prod_{k=i}^{d}\left|f_{k}(1)\right|\leq 1$ for all $1\leq i\leq d$, then the multiple additive Lucas $L$-function $\mathcal{L}_{U}^{d}(\bf s\mid \bf f)$ 
of depth $d$ can be continued to a meromorphic function on all of $\mathbb{C}^d$ with possible poles on the hyper planes 
\begin{align}\label{al3.5}
\begin{split}
s_j + \cdots + s_d &= -2(k_j+\cdots + k_d) + \frac{\log (f_j(1)\cdots f_d(1))}{\log \alpha}\\
& + (k_j + \cdots + k_d) \frac{\log |Q|}{\log \alpha} + \left(2n + \ell_{Q, k_j, \ldots, k_d}  \right) \frac{\pi i }{\log \alpha} 
\end{split}
\end{align}
for $1\leq j \leq d$ where $ k_j, n \in \mathbb{Z} $ with $ k_j \geq 0$.
\end{theorem}

\begin{proof}
 From \eqref{eq3.1},  for any $z \in \mathbb{C},$ we have
\[f(1)^{n} U_{n} ^{z}  = D^{-z/2} \sum_{k = 0} ^{\infty}\binom{z}{k}(-1)^{k} f(1)^{n} Q^{nk}\alpha^{n(z-2k)}.\]
Since, this series $ \mathcal{L}_{U}^{d}({\bf s}\mid {\bf f})$  is absolutely converges for a fixed point $(s_1, \ldots, s_d)$ in $\mathfrak{D}_d$,  Therefore, by interchanging the summation, we have 
\begin{align}\label{eq3.7}
\mathcal{L}_{U}^{d}({\bf s}\mid {\bf f})  &= D^{(s_1 +\cdots+ s_d)/2} \sum_{k_1 =0}^{\infty}\binom{-s_1}{k_1} (-1)^{k_1} \cdots  \sum_{k_d =0}^{\infty}\binom{-s_d}{k_d} (-1)^{k_d} \\ \nonumber
&\sum_{n_1, \ldots, n_d = 1}^{\infty} 
(f_1(1)\cdots f_d(1))^{n_1} Q^{(k_1+ \cdots k_d)n_1}\left( \alpha^{-(s_1 + \cdots +s_d + 2(k_1 + \cdots +k_d))}\right)^{n_1} \\&\times \cdots  \times ( f_d(1))^{n_d}Q^{k_dn_d}\left(\alpha^{-(s_d + 2k_d)} \right)^{n_d}.  \nonumber
\end{align}
Note that for any $(k_1, \ldots, k_d) \in \mathbb{Z}_{\geq 0}^{d},$ we have
\begin{align}\label{eq3.8}
\begin{split}
&\sum_{n_1, \ldots, n_d = 1}^{\infty}  \left(\frac{f_1(1)\cdots f_d(1)Q^{(k_1+ \cdots k_d)}}{\alpha^{s_1 + \cdots +s_d + 2(k_1 + \cdots +k_d)}}\right)^{n_1} \cdots  \left(\frac{f_d(1)Q^{k_d}}{\alpha^{s_d + 2k_d}} \right)^{n_d}\\
& =  \frac{f_1(1)\cdots f_d(1)Q^{k_1+\cdots + k_d} \alpha^{-(s_1 + \cdots + s_d + 2(k_1 +\cdots + k_d))}}{(1 -  f_1(1)\cdots f_d(1)Q^{k_1+\cdots + k_d}\alpha^{-(s_1 +\cdots + s_d + 2(k_1 + \cdots+k_d))})}  \cdots \frac{ f_d(1)Q^{k_d}\alpha^{-(s_d + 2k_d)}}{(1- f_d(1)Q^{k_d}\alpha^{-(s_d + 2k_d})}\\
& = \frac{f_1(1)\cdots f_d(1)Q^{k_1+\cdots + k_d}}{\alpha^{s_1 + \cdots + s_d + 2(k_1 +\cdots + k_d)} - f_1(1)\cdots f_d(1)Q^{k_1 + \cdots  +k_d}} \cdots  \frac{f_d(1)Q^{k_d} }{\alpha^{s_d + 2k_d} - f_d(d)Q^{k_d}}.
\end{split}
\end{align}
Using \eqref{eq3.8} in \eqref{eq3.7}, we obtain
\begin{align}\label{eq3.9}
\begin{split}
&\mathcal{L}_{U}^{d}({\bf s}\mid {\bf f}) = D^{\frac{s_1 +\cdots+ s_d}2} \sum_{k_1 =0}^{\infty}\binom{-s_1}{k_1} (-1)^{k_1} \cdots  \sum_{k_d =0}^{\infty}\binom{-s_d}{k_d} (-1)^{k_d}\\
& \frac{f_1(1)\cdots f_d(1)Q^{k_1+\cdots + k_d}}{\alpha^{s_1 + \cdots + s_d + 2(k_1 +\cdots + k_d)} - f_1(1)\cdots f_d(1)Q^{k_1 + \cdots  +k_d}} \cdots  \frac{f_d(1)Q^{k_d} }{\alpha^{s_d + 2k_d} - f_d(1)Q^{k_d}}.
\end{split}
\end{align}
The infinite series in \eqref{eq3.9} is a holomorphic function on  $\mathbb{C}^d$ except the poles derived from 
 the functions
\[ h_{k_j, \ldots, k_d}(s_j, \ldots, s_d) := \frac{f_j(1) \cdots f_d(1)Q^{k_j+\cdots + k_d }}{\alpha^{s_j+ \cdots + s_d + 2(k_j+\cdots + k_d)} - f_j(1)\cdots f_d(1)Q^{k_j + \cdots  +k_d}}.\]
\end{proof}

\section{Poles and residues of multiple Lucas Dirichlet series associated to additive and Dirichlet characters}\label{sec4}
For $1 \leq j \leq d$, we define 
\begin{align*}
s_d(j):= s_j &+ \cdots + s_d, \quad k_d(j) := k_j + \cdots + k_d,\;  k_d'(j) := k_j' + \cdots + k_d'   \\
& r_d(j):= r_j + \cdots + r_d, \quad  g_{ j}^{d} := \prod_{k= j}^{d} f_k,\quad \mbox{and} \quad \zeta_q = e^{\frac{2\pi i}{q}}
\end{align*}
with assumption that $s_d(d+1) =0, \;  k_d(d+1) =0,\; g_d^d =f_d \quad \hbox{and} \quad  k_d'(d+1) =0$.

\subsection{Residues of $\mathcal{L}_U^d({\bf s}\mid {\bf \chi})$}

The residue of the multiple Lucas $L$-function $\mathcal{L}_U^d({\bf s}\mid \chi)$ along these hyperplanes \eqref{eq3.6} is defined to be the restriction  of the meromorphic function
$$\left( s_d(j) + 2k_d(j) - k_d(j) \frac{\log |Q|}{\log \alpha} - \left(\frac{2n}{q} + \ell_{Q, k_j,\ldots,k_d}  \right) \frac{\pi i }{\log \alpha} \right)\mathcal{L}_U^d({\bf s} \mid \chi)$$ to the hyperplanes \eqref{eq3.6}.
 
 \begin{theorem}\label{th5.3}
 For a non-negative integer $k_d',$ let $a_d:= -2k_d'+ k_d' \frac{\log |Q|}{\log \alpha} + \left(\frac{2n}{q} + \ell_{Q, k_d'(d)}  \right) \frac{\pi i }{\log \alpha}$ and set $\zeta_U^{d-1}({\bf s}\mid {\bf r}) = 1$ for $d =1$. Then,
 \begin{align*}
 \underset{s_d = a_d}{\mathrm{Res}} \mathcal{L}_U^d( {\bf s }\mid \chi) &= \sum_{r_1 = 1}^{q }\chi_1(r_1) \sum_{r_2 = 1}^{q}\chi_{2}(r_2(1))\cdots \sum_{r_d=1}^{q}\chi_d(r_d(1))\\
& \times \zeta_U^{d-1}(s_1, \cdots, s_{d-1}\mid r_1, \cdots, r_{d-1}) D^{\frac{a_d}2} \binom{- a_d}{k_d'} \frac{(-1)^{k_d'} }{q\log \alpha}\zeta_{q}^{-nr_d(1)}. 
 \end{align*}
 \end{theorem}
 
\begin{proof}
The case $d = 1$ was done by Kamano \cite{Kamano2013}. Now
for $d > 1,$ when $s_d = a_d$, we have 
\begin{align*}
& (s_d +2k_d')\log \alpha = k_d' \log |Q|+ \left(\frac{2n}{q} + \ell_{Q, k_d'}  \right) \pi i 
\end{align*}
and this implies that $\alpha^{(s_d +2k_d')} =  Q^{ k_d'} \zeta_q^n$. Therefore, $\alpha^{-(r_d(1))(s_d + 2k_d')} - Q^{-(r_d(1))k_d'}\zeta_{q}^{-(r_d(1))n}$ is an analytic function with simple zeros at $ a_d$.  Thus 
\begin{align}\label{eq5.10}
\begin{split}
 \underset{s_d = a_d}{\mathrm{Res}} \frac{1}{1- Q^{qk_d}\alpha^{-q(s_d + 2k_d)}} &= \lim_{s_d \rightarrow a_d} \frac{s_d - a_d}{1- Q^{qk_d}\alpha^{-q(s_d + 2k_d)}} \\
&= \frac{1}{\frac{d}{ds_d}\left(1- Q^{qk_d}\alpha^{-q(s_d + 2k_d)}\right)}\Bigg|_{s_d = a_d} = \frac{1}{q\log \alpha}.
\end{split}
\end{align}

Hence, the residue of $\zeta_U^d({\bf s \mid {\bf r}})$  along the hyperplane $s_d=a_d$ is 
\begin{align*}
&\lim_{s_d \rightarrow a_d}  (s_d - a_d)D^{\frac{s_d(1)}2} \sum_{k_1 =0}^{\infty}\binom{-s_1}{k_1} (-1)^{k_1} \cdots \sum_{k_d =0}^{\infty}\binom{-s_d}{k_d} (-1)^{k_d} \times\\
&\times \frac{Q^{k_1r_1} \alpha^{-r_1(s_1 + 2k_1)}}{1 -  Q^{qk_d(1)}\alpha^{-q(s_d(1) + 2k_d(1))}}\times  \cdots \times \frac{ Q^{k_dr_d(1)}\alpha^{-r_d(1)(s_d + 2k_d)}}{(1- Q^{qk_d}\alpha^{-q(s_d + 2k_d)})}\\
&= D^{\frac{s_{d-1}(1)}2} \sum_{k_1 =0}^{\infty}\binom{-s_1}{k_1} (-1)^{k_1} \cdots \sum_{k_d =0}^{\infty}\binom{-s_{d-1}}{k_{d-1}} (-1)^{k_{d-1}} \\
&\times\frac{Q^{k_1r_1} \alpha^{-r_1(s_1 + 2k_1)}}{1 -  Q^{qk_d(1)}\alpha^{-q(s_d(1)+ 2k_d(1))}}  \Bigg|_{s_d= a_d}\times \cdots \times  \frac{ Q^{k_{d-1}r_{d-1}(1)}\alpha^{-r_{d-1}(1)(s_{d-1} + 2k_{d-1})}}{(1- Q^{qk_d(d-1)}\alpha^{-q(s_{d}(d-1) + 2k_{d}(d-1)})}\Bigg|_{s_d= a_d}\\
&\times \lim_{s_d \rightarrow a_d} D^{\frac{a_d}2}(-1)^{k_{d}} \binom{-s_d}{k_d}\frac{  (s_d - a_d)Q^{k_dr_d(1)}\alpha^{-r_d(1)(s_d + 2k_d)}}{(1- Q^{qk_d}\alpha^{-q(s_d + 2k_d})}.
\end{align*}

Note that after evaluation of limit, the terms containing only $k_d$ but not $k_1, \ldots, k_{d-1}$ survives when $k_d = k_d'$ and rest of the terms vanish. Hence, the above estimation becomes
\begin{align*}
 &\underset{s_d = a_d}{\mathrm{Res}} \zeta_U^d({\bf s \mid {\bf r}}) =  D^{\frac{a_d}2} D^{\frac{s_{d-1}(1)}2} \sum_{k_1 =0}^{\infty}\binom{-s_1}{k_1} (-1)^{k_1} \cdots \sum_{k_{d-1} =0}^{\infty}\binom{-s_{d-1}}{k_{d-1}} (-1)^{k_{d-1}} \\
 &\times \frac{Q^{k_1r_1} \alpha^{-r_1(s_1+2k_1)}}{1 -  Q^{qk_{d-1}(1)}\alpha^{-q(s_{d-1}(1) + 2k_{d-1}(1))}}\times   \cdots  \times  \frac{ Q^{k_{d-1}r_{d-1}(1)}\alpha^{-r_{d-1}(1)(s_{d-1} + 2k_{d-1})}}{(1- Q^{qk_{d-1}(d-1)}\alpha^{-q(s_{d-1}(d-1) + 2k_{d-1}(d-1)})}\\
&\times  \frac{1}{q\log \alpha} \binom{- a_d}{k_d'} (-1)^{k_d'} \zeta_{q}^{-nr_d(1)}\\
&=  \zeta_U^{d-1}(s_1, \cdots, s_{d-1}\mid r_1, \cdots, r_{d-1}) D^{\frac{a_d}2} \binom{- a_d}{k_d'} (-1)^{k_d'} \frac{ \zeta_{q}^{-nr_d(1)}}{q\log \alpha}.
\end{align*}
Therefore, 
\begin{align}\label{eq86}
\begin{split}
 \underset{s_d = a_d}{\mathrm{Res}}\mathcal{L}_U^d({\bf s}\mid \chi) &= \sum_{r_1 = 1}^{q}\chi_1(r_1) \sum_{r_2 = 1}^{q}\chi_{2}(r_2(1))\cdots \sum_{r_d=1}^{q}\chi_d(r_d(1)) \underset{s_d = a_d}{\mathrm{Res}} \zeta_U^d({\bf s \mid {\bf r}})\\
 &= \sum_{r_1 = 1}^{q}\chi_1(r_1) \sum_{r_2 = 1}^{q}\chi_{2}(r_2(1))\cdots \sum_{r_d=1}^{q}\chi_d(r_d(1))\zeta_{q}^{-nr_d(1)}\\
 & \times\zeta_U^{d-1}(s_1, \cdots, s_{d-1}\mid r_1, \cdots, r_{d-1}) D^{\frac{a_d}2} \binom{- a_d}{k_d'} \frac{(-1)^{k_d'} }{q\log \alpha}.\\
 \end{split}
\end{align}
\end{proof}  
For $n\in \mathbb{Z}$, we define the Gauss sum $\tau(\chi, n)$ by the formula
\[\tau(\chi,n) = \sum_{x \pmod q} \chi(x)\zeta_q^{nx}.\]
 For fixed $(r_1, \cdots, r_{d-1}) \in \mathbb{Z}_{>0}^{d-1}$, we have 
 \begin{equation}\label{eq87}
 \sum_{r_d=1}^{q}\chi_d(r_d(1)) \zeta_{q}^{-nr_d(1)} = \sum_{r_d=1}^{q} \chi_d(-1)\chi_d(-nr_d(1)) \zeta_{q}^{-nr_d(1)} = \chi_d(-1) \tau(\chi_d,n)
 \end{equation}
 as $\chi_d$ is periodic modulo $q$. We state the following lemma related to Gauss sum from \cite[Proposition 2.1.40]{Cohen2007}.
 \begin{lemma}\label{lem89}
For $d = \gcd(a, q),$ if we assume that $\chi$ can not be defined modulo $q/d$, then $\tau(\chi,a) = 0$.
 \end{lemma}
 It is known that the Dirichlet $L$-function $L(s, \chi)$ defined in \eqref{eq5} is holomorphic on the whole complex plane when $\chi$ is non-principal. Furthermore, the function $L(s, \chi)$ has trivial zeros at non-positive integers. For the function $\mathcal{L}_{U}^{d}({\bf s}\mid \chi),$ we have the following result.
\begin{coro}
Let $\chi_d$ be non-principal character modulo $q$. Then the function  $\mathcal{L}_{U}^{d}({\bf s}\mid \chi)$ is holomorphic on the real axis of $s_d$.
\end{coro}
\begin{proof}
By Theorem \ref{thm2}, the series $\mathcal{L}_U^d(\bf s\mid \chi)$ is holomorphic except the singularities given by \eqref{eq3.6}. If $s_d(d)$ is a real number then, from \eqref{eq3.6}, we must have $\left(2n/q + \ell_{Q, k_d(d)}\right)  = 0$. Since $2n/q+ \ell_{Q, k_d'(d)}\in \mathbb{Z}$, we have $q\mid n$ or $(q/2)\mid n$ according as $q$ is odd or even and this implies $\gcd(q, n) = q$ or $q/2$.  Thus $q/\gcd(q, n) = 1$ or $2$. Since $\chi_d$ is non-principal, the character $\chi_d$ cannot be defined for mod $1$ or $2$. Using Lemma \ref{lem89}, we conclude that $\tau(\chi_d, n) = 0$. Hence, our claim follows.
\end{proof}
Finally, we compute the residues of $\mathcal{L}_U^d({\bf s}\mid \chi)$ along the other hyperplanes \eqref{eq3.6} for $1\leq j \leq d-1$.  
\begin{theorem}\label{th5.4a}
 Let  $d \geq 2$ be a positive integer and $j$ be a positive integer with $1 \leq j \leq d-1$.  Let $k_j', \ldots,k_d'$ be non-negative integers with
 $$a_d(j) := -2k_d'(j) + k_d'(j)  \frac{\log |Q|}{\log \alpha} + \left(\frac{2n}{q} + \ell_{Q, k_d'(j)}  \right) \frac{\pi i }{\log \alpha}.$$ Then
 \begin{align*}
 \underset{s_d(j) = a_d(j)}{\mathrm{Res}}\mathcal{L}_{U}^{d}({\bf s}\mid \chi) &= D^{\frac{a_d(j)}2}  \zeta_{U}^{j-1}({\bf s}\mid \chi) \sum_{\substack{k_j \geq 0,\ldots,k_d \geq 0\\ k_d(j) =  k_d'(j)} }\binom{-s_j}{k_j} (-1)^{k_j} \cdots \binom{-s_d}{k_d} (-1)^{k_d}\\
&  \times \sum_{r_1= 1}^{q }\chi_1(r_1(1)) \cdots \sum_{r_d=1}^{q}\chi_d(r_d(1))\\
&\times\frac{Q^{k_d(j+1)r_{j+1}} \alpha^{-r_{j+1}(s_d(j+1)+ 2k_d(j+1))}}{1 -  Q^{qk_d(j+1)}\alpha^{-q(s_d(j+1) + 2k_d(j+1))}} \cdots  \frac{ Q^{k_dr_d}\alpha^{-r_d(s_d + 2k_d)}}{(1- Q^{qk_d}\alpha^{-q(s_d + 2k_d})} \frac{\zeta_q^{-nr_j(1)}}{q\log \alpha}.
 \end{align*}
\end{theorem}

\begin{proof}
Note that when $s_d(j) = a_d(j), $ we have 
\begin{align*}
\begin{split}
&  \left( s_d(j)+ 2k_d'(j)\right) \log \alpha= k_d'(j) \log |Q|+ \left(\frac{2n}{q} + \ell_{Q, k_d'(j)}  \right) \pi i
\end{split}
\end{align*}
and this implies that $\alpha^{ s_d(j)+2k_d'(j)} =  Q^{k_d'(j)} \zeta_q^n$. Therefore, for $ 1 \leq \ell \leq j,$ we obtain 
\begin{align}\label{eq5.10a}
\begin{split}
 \alpha^{-r_{\ell}\left(s_d(j)+ 2k_d'(j)\right)}  &=  Q^{-r_{\ell} k_d'(j)} \zeta_{q}^{-n r_{\ell}},\quad  \alpha^{-r_j(1)\left( s_d(j)+ 2k_d'(j)\right)}  =  Q^{-r_j(1) k_d'(j)} \zeta_{q}^{-n r_j(1)}\\
&   \alpha^{-q \left( s_d(j)+2k_d'(j)\right)} =  Q^{-qk_d'(j)}.
 \end{split}
\end{align} 
 Further, we observe that $ \alpha^{q((s_d(j)+ 2k'_d(j))} -  Q^{qk'_d(j)}$ is an analytic function with simple zeros at $ a_d(j)$ when $1\leq j< d$.  By proceeding as in the proof of Theorem \ref{th5.3}, we have
\begin{align*}
& \underset{s_d(j) = a_{d}(j)}{\mathrm{Res}} \frac{1}{1 -  Q^{qk_d(j)}\alpha^{-q(s_d(j) + 2k_d(j))}}  = \lim_{s_d(j) \rightarrow a_{d}(j)} \frac{\left( s_d(j) - a_{d}(j)\right)}{1 -  Q^{qk_d(j)}\alpha^{-q(s_d(j) + 2k_d(j))}}
 = \frac{1}{q\log \alpha}.
\end{align*}
Now consider the following limit:
\begin{align}\label{eq5.10e}
\begin{split}
& \lim_{s_d(j) \rightarrow a_{d}(j)} D^{\frac{s_d(j)}2}  \sum_{k_j, \ldots,k_d =0}^{\infty}\binom{-s_j}{k_j} (-1)^{k_j} \cdots \binom{-s_d}{k_d} (-1)^{k_d} \\
&\frac{ Q^{k_d(j)r_j} \alpha^{-r_j(s_d(j)+ 2k_d(j))} \left( s_d(j) - a_{d}(j)\right)}{1 -  Q^{qk_d(j)}\alpha^{-q(s_d(j) + 2k_d(j))}}  
\cdots \frac{ Q^{k_dr_d}\alpha^{-r_d(s_d + 2k_d)}}{(1- Q^{qk_d}\alpha^{-q(s_d + 2k_d})}.
\end{split}
\end{align}
Note that after the calculation of the limit, only those terms containing $k_j, \ldots, k_d$ will survive when $k_d(j) = k_d'(j)$ and rest of the terms will vanish. Therefore \eqref{eq5.10e} becomes
\begin{align*}
&D^{\frac{a_d(j)} 2} \sum_{\substack{k_j \geq 0,\ldots,k_d \geq 0\\ k_d(j) =  k_d'(j)} }\binom{-s_j}{k_j} (-1)^{k_j} \cdots \binom{-s_d}{k_d} (-1)^{k_d} \times\\
&\frac{ \zeta_{q}^{-n r_{j}} Q^{k_d(j+1)r_{j+1}} \alpha^{-r_{j+1}(s_d(j+1)+ 2k_d(j+1))}}{1 -  Q^{qk_d(j+1)}\alpha^{-q(s_d(j+1) + 2k_d(j+1))}} \cdots  \frac{ Q^{k_dr_d}\alpha^{-r_d(s_d + 2k_d)}}{(1- Q^{qk_d}\alpha^{-q(s_d + 2k_d})}\frac{1}{q\log \alpha}.
\end{align*}

Hence, the residue of $\zeta_U^d({\bf s} \mid {\bf r})$  along the hyperplane $s_d(j) = a_d(j) $ is 
\begin{align*}
  \underset{s_d(j) = a_d(j)}{\mathrm{Res}} & \zeta_U^d({\bf s} \mid {\bf r})= \lim_{s_d(j) \rightarrow a_{d}(j) }  \left( s_d(j) - a_{d}(j)\right) \zeta_U^d({\bf s} \mid {\bf r}) \\
  & =  D^{\frac{s_d(j)}2} D^{\frac{s_{j-1}(1)}{2}} \sum_{k_1, \dots,k_{j-1} =0}^{\infty}\binom{-s_1}{k_1} (-1)^{k_1} \cdots \binom{-s_{j-1}}{k_{j-1}} (-1)^{k_{j-1}} \times\\
 & \sum_{k_j, \ldots,k_d =0}^{\infty}\binom{-s_j}{k_j} (-1)^{k_j} \cdots \binom{-s_d}{k_d} (-1)^{k_d}\frac{Q^{k_d(1)r_{1}} \alpha^{-r_{1}(s_d(1)+ 2k_d(1))}}{1 -  Q^{qk_d(1)}\alpha^{-q(s_d(1) + 2k_d(1))}}\bigg|_{s_d(j) = a_d(j)}\\ 
 &\times \cdots \times \frac{Q^{k_d(j-1)r_{j-1}} \alpha^{-r_{j-1}(s_d(j-1)+ 2k_d(j-1))}}{1 -  Q^{qk_d(j-1)}\alpha^{-q(s_d(j-1) + 2k_d(j-1))}}\bigg|_{s_d(j) = a_d(j)}\\
 &\times  \lim_{s_d(j) \rightarrow a_{d}(j) }   \frac{Q^{k_d(j)r_j} \alpha^{-r_j(s_d(j)+ 2k_d(j))} \left( s_d(j) - a_{d}(j)\right)}{1 -  Q^{qk_d(j)}\alpha^{-q(s_d(j) + 2k_d(j))}} \cdots  \frac{ Q^{k_dr_d}\alpha^{-r_d(s_d + 2k_d)}}{(1- Q^{qk_d}\alpha^{-q(s_d + 2k_d})}\\
=& D^{\frac{a_d(j)} 2} \zeta_U^{j-1}(s_1, \ldots, s_{j-1}; r_1, \ldots, r_{j-1}) \sum_{\substack{k_j \geq 0,\ldots,k_d \geq 0\\ k_d(j) =  k_d'(j)} }\binom{-s_j}{k_j} (-1)^{k_j} \cdots \binom{-s_d}{k_d} (-1)^{k_d} \\
&\times\frac{Q^{k_d(j+1)r_{j+1}} \alpha^{-r_{j+1}(s_d(j+1)+ 2k_d(j+1))}}{1 -  Q^{qk_d(j+1)}\alpha^{-q(s_d(j+1) + 2k_d(j+1))}} \cdots  \frac{ Q^{k_dr_d}\alpha^{-r_d(s_d + 2k_d)}}{(1- Q^{qk_d}\alpha^{-q(s_d + 2k_d})} \frac{\zeta_q^{-nr_j(1)}}{q\log \alpha}.
\end{align*}
Hence, the Theorem \ref{th5.4a} follows from \eqref{eq3.6d}.
\end{proof}

\subsection{Residues of $\mathcal{L}_U^d({\bf s}\mid {\bf f})$}

The residue of the multiple additive Lucas $L$-function $\mathcal{L}_U^d({\bf s} \mid {\bf f})$ along these hyperplanes \eqref{al3.5} is defined to be the restriction  of the meromorphic function
$$\left( s_d(j) +2k_d(j) -\frac{\log g_j^d(1)}{\log \alpha} - k_d(j) \frac{\log |Q|}{\log \alpha} - \left(2n + \ell_{Q, k_j,\ldots,k_d}  \right) \frac{\pi i }{\log \alpha} \right)\mathcal{L}_U^d({\bf s} \mid {\bf f})$$ to the hyperplanes \eqref{al3.5}.

 \begin{theorem}\label{th5.5}
 For a non-negative  integer $k_d',$  let $b_d:= -2k_d'+ \frac{\log f_d(1)}{\log \alpha} + k_d'\frac{\log |Q|}{\log \alpha} + \left(2n + \ell_{Q, k_d'}  \right) \frac{\pi i }{\log \alpha}$. Then  for $d >1$,
 $$
 \underset{s_d = b_d}{\mathrm{Res}} \mathcal{L}_U^{d}({\bf s} \mid {\bf f}) =  \mathcal{L}_U^{d-1}({\bf s} \mid {\bf f})D^{\frac{b_d}2} \binom{- b_d}{k_d'} (-1)^{k_d'} \frac{1}{\log \alpha}.$$
 \end{theorem}
 
\begin{proof}
Note that 
\begin{align*}
\begin{split}
 \underset{s_d = b_d}{\mathrm{Res}} \frac{1}{\alpha^{s_d + 2k_d'} - Q^{k_d'}f_d(1)} &= \lim_{s_d \rightarrow b_d} \frac{s_d - b_d}{\alpha^{s_d + 2k_d'} - Q^{k_d'}f_d(1)} \\
&= \frac{1}{\frac{d}{ds_d}\left(\alpha^{s_d + 2k_d'} - Q^{k_d'f_d(1)}\right)}\Bigg|_{s_d = b_d} = \frac{1}{Q^{k_d'}f_d(1)\log \alpha}.
\end{split}
\end{align*}
Hence, the residue of $\mathcal{L}_U^d({\bf s} \mid {\bf f})$  along the hyperplane $b_d(d)$ is 
\begin{align*}
&\lim_{s_d \rightarrow b_d(d)}  (s_d - b_d)D^{\frac{s_d(1)}2} \sum_{k_1 =0}^{\infty}\binom{-s_1}{k_1} (-1)^{k_1} \cdots \sum_{k_d =0}^{\infty}\binom{-s_d}{k_d} (-1)^{k_d}\\
& \times \frac{g_1^d(1)Q^{k_d(1)}}{\alpha^{s_d(1) + 2k_d(1)} - g_1^d(1)Q^{k_d(1)}}\cdots  \frac{f_d(1)Q^{k_d} }{\alpha^{s_d + 2k_d} - f_d(1)Q^{k_d}}\\
&= D^{\frac{s_{d-1}(1)}2} \sum_{k_1 =0}^{\infty}\binom{-s_1}{k_1} (-1)^{k_1} \cdots \sum_{k_d =0}^{\infty}\binom{-s_{d-1}}{k_{d-1}} (-1)^{k_{d-1}} \times
 \frac{g_1^d(1)Q^{k_d(1)}}{\alpha^{s_d(1) + 2k_d(1) } - g_1^d(1)Q^{k_d(1) }} \Bigg|_{s_d= b_d}\\
 & \times \frac{g_2^d(1)Q^{k_d(2)}}{\alpha^{s_d(2) + 2k_d(2)} - g_2^d(1)Q^{k_d(2)}} \Bigg|_{s_d= b_d}
\cdots  \lim_{s_d \rightarrow b_d} D^{\frac{b_d}2}(-1)^{k_{d}} \binom{-s_d}{k_d}\frac{f_d(1)Q^{k_d} (s_d - b_d) }{\alpha^{s_d + 2k_d} -f_d(1) Q^{k_d}}.
\end{align*}
Now,
\begin{align*}
 \underset{s_d = b_d}{\mathrm{Res}} \mathcal{L}_U^d({\bf s} \mid {\bf f})  &=  D^{\frac{b_d}2} D^{\frac{s_{d-1}(1)}2} \sum_{k_1 =0}^{\infty}\binom{-s_1}{k_1} (-1)^{k_1} \cdots \sum_{k_{d-1} =0}^{\infty}\binom{-s_{d-1}}{k_{d-1}} (-1)^{k_{d-1}} \times\\
&\frac{g_1^{d-1}(1)Q^{k_{d-1}(1)}}{\alpha^{s_{d-1}(1)  + 2k_{d-1}(1)} - g_1^{d-1}(1)Q^{k_{d-1}(1) }}  \cdots \frac{1}{\log \alpha} \binom{-b_d}{k_d'} (-1)^{k_d'}\\
&= \mathcal{L}_U^{d-1}({\bf s} \mid {\bf f})D^{\frac{b_d}2} \binom{- b_d}{k_d'} (-1)^{k_d'} \frac{1}{\log \alpha}.
\end{align*}
\end{proof}  
Finally, we compute the residues along the other hyperplanes \eqref{al3.5} for $1\leq j \leq d-1$.  
\begin{theorem}\label{th5.4}
For a positive integer $d \geq 2$ and for a positive integer $j$ with $1 \leq j \leq d-1,$ let $k_j', \ldots,k_d'$ be non-negative integers. Let 
 $$b_d(j) := -2k_d'(j)+\frac{\log g_j^d(1)}{\log \alpha} + k_d'(j) \frac{\log |Q|}{\log \alpha} + \left(2n + \ell_{Q, k_j',\ldots, k_d'}  \right) \frac{\pi i }{\log \alpha}.$$ Then
\begin{align*}
 \underset{s_d(j) = b_d(j)}{\mathrm{Res}} \mathcal{L}_U^d({\bf s} \mid {\bf f}) & = D^{\frac{b_d(j)}2} \mathcal{L}_U^{j-1}({\bf s} \mid {\bf f})\sum_{\substack{k_j \geq 0,\ldots,k_d \geq 0\\ 
 k_d(j) =  k_d(j)'}}\binom{-s_j}{k_j} (-1)^{k_j} \cdots \binom{-s_d}{k_d}  \\
& \times (-1)^{k_d}  \frac{g_{j+1}^d(1)Q^{k_d(j+1)}}{\alpha^{s_d(j+1)+ 2k_d(j+1)} - Q^{k_d(j+1)}} \cdots \frac{f_d(1)Q^{k_d}}{\alpha^{s_d + 2k_d} - Q^{k_d}}  \frac{1}{\log \alpha}.
\end{align*}
\end{theorem}

\begin{proof}

By proceeding as in the proof of Theorem \ref{th5.3}, we have
\begin{align*}
 \lim_{s_d(j) \rightarrow b_{d}(j) }  \frac{ g_j^d(1)Q^{k_d(j)}\left( s_d(j) - b_{d}(j)\right)}{\alpha^{s_d(j)+ 2k_d(j)} - g_j^d(1)Q^{k_d(j)}}  = & \quad  g_j^d(1)Q^{k_d(j)} \underset{s_d(j) - b_{d}(j)}{\mathrm{Res}} \frac{1}{\alpha^{s_d(j)+ 2k_d(j) } - g_j^d(1)Q^{k_d(j)}} \\
 &= \frac{1}{\log \alpha}.
\end{align*}
Note that
\begin{align*}\label{eq5.10b}
\lim_{s_d(j) \rightarrow b_{d}(j)} & D^{\frac{s_d(j)}2}  \sum_{k_j, \ldots,k_d =0}^{\infty}\binom{-s_j}{k_j} (-1)^{k_j} \cdots \binom{-s_d}{k_d} (-1)^{k_d} \\
& \frac{ \left( s_d(j) - b_{d}(j)\right) g_j^d(1)Q^{k_d(j)}}{\alpha^{s_d(j)  + 2k_d(j)} - g_j^d(1)Q^{k_d(j) }}  \cdots \frac{f_d(1)Q^{k_d(d)}}{\alpha^{s_d(d)  + 2k_d(d)} - f_d(1)Q^{k_d(d)}}.\\
&= D^{\frac{b_d(j)} 2} \sum_{\substack{k_j \geq 0,\ldots,k_d \geq 0\\ k_d(j) =  k_d'(j)} }\binom{-s_j}{k_j} (-1)^{k_j} \cdots \binom{-s_d}{k_d} (-1)^{k_d} \\
&\times \frac{g_{j+1}^d(1)Q^{k_d(j+1)}}{\alpha^{s_d(j+1) + 2k_d(j+1)} - g_{j+1}^d(1)Q^{k_d(j+1)}} \cdots \frac{f_d(1)Q^{k_d}}{\alpha^{s_d + 2k_d} - f_d(1)Q^{k_d}}  \frac{1}{\log \alpha}.
\end{align*}
Hence, the residue of $\mathcal{L}_U^d({\bf s} \mid {\bf f})$  along the hyperplane $s_d(j) = b_d(j) $ is 
\begin{align*}
  \underset{s_d(j) = b_d(j)}{\mathrm{Res}} & \mathcal{L}_U^d({\bf s} \mid {\bf f}) = \lim_{s_d(j) \rightarrow b_{d}(j) }  \left( s_d(j) - b_{d}(j)\right) \mathcal{L}_U^d({\bf s} \mid {\bf f}) \\
 &=  D^{\frac{s_d(j)}2} D^{\frac{s_1+ \cdots+ s_{j-1}}2} \sum_{k_1, \dots,k_{j-1} =0}^{\infty}\binom{-s_1}{k_1} (-1)^{k_1} \cdots \binom{-s_{j-1}}{k_{j-1}} (-1)^{k_{j-1}} \times\\
& \sum_{k_j, \ldots,k_d =0}^{\infty}\binom{-s_j}{k_j} (-1)^{k_j} \cdots \binom{-s_d}{k_d} (-1)^{k_d}
\frac{g_1^d(1)Q^{k_d(1)}}{\alpha^{s_d(1)  + 2k_d(1)} - g_1^d(1)Q^{k_d(1) }}\bigg|_{s_d(j) = b_d(j)}  \cdots \\
& \frac{g_{j-1}^d(1)Q^{k_d(j-1)}}{\alpha^{s_d(j-1)  + 2k_d(j-1)} - g_{j-1}^d(1)Q^{k_d(j-1)}} \bigg|_{s_d(j) = b_d(j)} \lim_{s_d(j) \rightarrow b_{d}(j) } \frac{ \left( s_d(j) - b_{d}(j)\right) g_j^d(1)Q^{k_d(j)}}{\alpha^{s_d(j)  + 2k_d(j)} - g_j^d(1)Q^{k_d(j) }} \\
& \cdots \frac{f_d(1)Q^{k_d}}{\alpha^{s_d  + 2k_d} - f_d(1)Q^{k_d}} \\
=&  D^{\frac{s_d(j)}2} D^{\frac{s_1+ \cdots+ s_{j-1}}2} \sum_{k_1, \dots,k_{j-1} =0}^{\infty}\binom{-s_1}{k_1} (-1)^{k_1} \cdots \binom{-s_{j-1}}{k_{j-1}} (-1)^{k_{j-1}} \times\\
& \frac{g_1^{j-1}(1)Q^{k_{j-1}(1)}}{\alpha^{s_{j-1}(1)  + 2k_{j-1}(1)} - g_1^{j-1}(1)Q^{k_{j-1}(1) }}  \cdots 
 \frac{g_{j-1}^{j-1}(1)Q^{k_{j-1}(j-1)}}{\alpha^{s_{j-1}(j-1)  + 2k_{j-1}(j-1)} - g_{j-1}^{j-1}(1)Q^{k_{j-1}(j-1)}}   \\ & \sum_{k_j, \ldots,k_d =0}^{\infty}\binom{-s_j}{k_j} (-1)^{k_j} \cdots \binom{-s_d}{k_d} (-1)^{k_d}   \lim_{s_d(j) \rightarrow b_{d}(j) } \frac{ \left( s_d(j) - b_{d}(j)\right) g_j^d(1)Q^{k_d(j)}}{\alpha^{s_d(j)  + 2k_d(j)} - g_j^d(1)Q^{k_d(j) }} \\
& \cdots \frac{f_d(1)Q^{k_d(d)}}{\alpha^{s_d(d)  + 2k_d(d)} - f_d(1)Q^{k_d(d)}} \\
=& D^{\frac{b_d(j)} 2} \mathcal{L}_U^{j-1}({\bf s} \mid {\bf f}) \sum_{\substack{k_j \geq 0,\ldots,k_d \geq 0\\ k_d(j) =  k_d'(j)} }\binom{-s_j}{k_j} (-1)^{k_j} \cdots \binom{-s_d}{k_d} (-1)^{k_d} \times\\
& \frac{g_{j+1}^d(1)Q^{k_d(j+1)}}{\alpha^{s_d(j+1) + 2k_d(j+1)} - g_{j+1}^d(1)Q^{k_d(j+1)}} \cdots \frac{f_d(1)Q^{k_d}}{\alpha^{s_d + 2k_d} - f_d(1)Q^{k_d}}  \frac{1}{\log \alpha}.
\end{align*}
\end{proof}

\section{Special values at negative integers}\label{sec5}
     
     In this section, the special values of $\mathcal{L}_{U}^{d}({\bf s}\mid \chi) $ at negative integer arguments are discussed. First, we give a sufficient condition for $\mathcal{L}_{U}^{d}({\bf s}\mid \chi) $  to be holomorphic at $(s_1, \ldots, s_d) = (-m_1, \ldots, -m_d)$ where $ m_1, \ldots, m_d \in \mathbb{N}$. Further, we denote \[m_d(j) = m_j + m_{j+1} + \cdots + m_d, \; \mbox{for}\;\; j = 1, 2, \ldots, d.\]
\begin{lemma}[See Lemma 5, \cite{Kamano2013}]\label{lem3}
Let $P$ and $Q$ be rational numbers and $\sqrt{D}$ is an irrational number. Then $\log|Q|/ \log\alpha$ is rational if and only if $Q = \pm 1.$
\end{lemma}
\begin{proposition}\label{prop2}
Let $(m_1, \ldots, m_d) \in \mathbb{N}^d$ and $\chi$ be a Dirichlet character of modulus $q,$ where $q$ is a positive integer.
\begin{enumerate}
\item
For $P \in \mathbb{Q}$ and $Q = 1,$ the function  $\mathcal{L}_{U}^{d}({\bf s}\mid \chi) $ is holomorphic at $ (s_1, \ldots, s_d)= (-m_1, \ldots, -m_d)$ if and only if  
\[m_d(1) \not\equiv 0\ (\textrm{mod}\ 2), m_d(2) \not\equiv 0\ (\textrm{mod}\ 2),\cdots, m_d \not\equiv 0\ (\textrm{mod}\ 2).\]
\item
For $P \in \mathbb{Q}$ and $Q = -1,$ the function $\mathcal{L}_{U}^{d}({\bf s}\mid \chi) $ is holomorphic at $ (s_1, \ldots, s_d)= (-m_1, \ldots, -m_d)$ if and only if 
\[qm_d(1) \not\equiv 0\ (\textrm{mod}\ 4), qm_d(2) \not\equiv 0\ (\textrm{mod}\ 4),\cdots, qm_d \not\equiv 0\ (\textrm{mod}\ 4).\]

\item
For $P, Q \in \mathbb{Q}$ with $Q \not= \pm 1,$  and $\sqrt{D}$ is an irrational number, the function  $\mathcal{L}_{U}^{d}({\bf s}\mid \chi) $ is holomorphic at $ (s_1, \ldots, s_d)= (-m_1, \ldots, -m_d)$ for all $(m_1, \ldots, m_d) \in \mathbb{N}^d.$
\end{enumerate}
\end{proposition}
\begin{proof}
 Let $Q= 1.$ Then the infinite series \eqref{eq3.3} is holomorphic except the poles derived from 
 \begin{equation}\label{eq51}
 \left(\alpha^{q(s_d(1) + 2k_d(1))} - 1 \right) \times \cdots \times \left(\alpha^{q(s_d(d) + 2k_d(d))} - 1 \right) = 0.
 \end{equation} 
 Note that \eqref{eq51} is true if and only if one of the following equation holds  \[s_d(1) = -2k_d(1), s_d(2) = -2k_d(2), \cdots, s_d(d) = -2k_d(d)\]
 for $(k_1, \ldots,k_d) \in \mathbb{N}^d$. Therefore the claim of $(1)$ follows.
 
Similarly, when $ Q= -1$, the infinite series \eqref{eq3.3} is holomorphic except the poles derived from
 \begin{equation}\label{eq52}
 \left(\alpha^{q(s_d(1) + 2k_d(1)) } - (-1)^{q k_d(1)} \right) \times \cdots \times \left(\alpha^{q(s_d(d) + 2k_d(d))} - (-1)^{qk_d(d)} \right) = 0.
 \end{equation} 
Again \eqref{eq52} is true if and only if one of the following equation holds  
\[qs_d(1) = -2qk_d(1), qs_d(2) = -2qk_d(2), \cdots, qs_d(d) = -2qk_d(d)\]
   with $ qk_d(1) \equiv 0 \pmod{2}, qk_d(2) \equiv 0 \pmod{2}, \ldots, qk_d(d) \equiv 0 \pmod{2}$.
   So the statement $(2)$ holds. 
%

   Note that  $ (s_1, \ldots, s_d)= (-m_1, \ldots, -m_d)$ is not a pole for all $(m_1, \ldots, m_d) \in \mathbb{N}^d,$ when $Q \neq \pm 1,$  and $\sqrt{D}$ is an irrational number, due to Lemma \ref{lem3}.  Thus, the statement $(3)$ holds.
 \end{proof}
Now onwards, we define the following notations to write our expression in simpler form. For $1\leq j \leq d$, let
 \begin{align}\label{eq5.55}
 \begin{split}
 \gamma_j ( k_j, \ldots, k_d;r_j) &:= (-1)^{k_j} \frac{Q^{r_j(k_d(j))}\alpha^{-r_j(-m_d(j) + 2k_d(j))}}{1 -  Q^{q(k_d(j))}\alpha^{-q(-m_d(j) + 2k_d(j))}} .
\end{split}
\end{align}
Now, by replacing $k_j$ by $m_j - k_j$ in \eqref{eq5.55}, we have the following. 
\begin{align}\label{eq5.55a}
\begin{split}
\gamma_j ( \overline{k_j}, \ldots, k_d; r_j) &:= (-1)^{m_j -k_j} \frac{Q^{r_j((m_j - k_j + k_{d}(j+1))}\alpha^{-r_j(-m_d(j) + 2(m_j - k_j + k_{d}(j+1))}}{1 -  Q^{q((m_j - k_j + k_{d}(j+1))}\alpha^{-q(-m_d(j) + 2(m_j - k_j+ k_{d}(j+1))}}.
\end{split}
 \end{align}
Further, we denote
 \begin{align}\label{eq5.57}
\begin{split}
\sigma_0&(k_1, \ldots k_d) =  \gamma_1(k_1,\ldots, k_d; r_1)\times  \gamma_2(k_2,\ldots, k_d;r_2) \times \cdots \times \gamma_d(k_d; r_d)\\
\sigma_{q}&(k_1, \ldots,\overline{k_{i_1}}, \ldots, \overline{k_{i_q}}, \ldots, k_d) =  \gamma_1(k_1, \ldots,\overline{k_{i_1}}, \ldots, \overline{k_{i_q}}, \ldots, k_d; r_1)\\
&\times \cdots \times \gamma_{i_1}(\overline{k_{i_1}}, \ldots, \overline{k_{i_q}}, \ldots, k_d;r_{i_1})
\gamma_{i_2}(\overline{k_{i_2}}, \ldots, \overline{k_{i_q}}, \ldots, k_d; r_{i_2} )\\
&\times \cdots\times \gamma_{i_q}(\bar{k_{i_q}},\ldots, k_d; r_{i_q})\times \cdots\times \gamma_{d}(k_{d}; r_d).
\end{split}
\end{align}

Note  that $\sigma_{q}(k_1, \ldots,\overline{k_{i_1}}, \ldots, \overline{k_{i_q}}, \ldots, k_d)$ defines that, $q$ number of integers in the tuple $(k_1,\ldots, k_d)$ changes from $k_{t}$ to $m_{t}-k_{t}$ for $1\leq t \leq q.$ 
 The following proposition is very crucial to prove Proposition \ref{prop14}.
\begin{proposition}\label{prop13}
Let $\psi$ be the non-trivial automorphism of $\mbox{Gal}(\mathbb{Q}(\sqrt{D})/ \mathbb{Q})$ and $\sigma_q$ as in \eqref{eq5.57}.
 Then, for any $ 0 \leq q \leq d,$ we obtain 
 \begin{align}\label{eq5.59}
 \begin{split}
 \psi(\sigma_{q}(k_1, &\ldots,\overline{k_{i_1}}, \ldots, \overline{k_{i_q}}, \ldots, k_d))\\ &= (-1)^{m_d(1)} (\sigma_{d-q}(\overline{k_1}, \ldots,\overline{k_{i_1-1}},k_{i_1}, \ldots,k_{i_q}, \overline{k_{i_q +1}}, \ldots, \overline{k_d})).
 \end{split}
   \end{align}
  \end{proposition}
  \begin{proof}
  We omit the proof as it follows from \cite[Proposition 6]{Meher2018} with some minor modifications.
  \end{proof}

  We will prove that the rationality of $\zeta_{U}^{d}(\bf  -{\bf m }\mid \bf r)$ in the following Proposition.
 \begin{proposition}\label{prop14}
 For a positive integer $d \geq 1$ let ${\bf m} = (m_1, \ldots, m_d) \in \mathbb{N}^d$.  Let $P, Q \in \mathbb{Q}$ and $\sqrt{D}$ be an irrational number. Then $\zeta_U^d(\bf -{\bf m}\mid {\bf r})$ is rational except for singularities.
 \end{proposition}
 
 \begin{proof}
Note that
  \begin{align}\label{eq5.54}
 \begin{split}
 \zeta_{U}^{d}&(\bf  -{\bf m }\mid \bf r)= \\
& = D^{-m_d(1)/2} \sum_{k_1 =0}^{m_1}\binom{m_1}{k_1} (-1)^{k_1} \times \cdots \times \sum_{k_{d-1} =0}^{m_{d-1}}\binom{m_{d-1}}{k_{d-1}} (-1)^{k_{d-1}}\\
&\frac{1}{2}\bigg[ \sum_{k_d =0}^{m_d}\binom{m_d}{k_d} (-1)^{k_d} \times \frac{Q^{k_d(1)r_1} \alpha^{-r_1(- m_d(1) + 2(k_d(1)))}}{1 -  Q^{qk_d(1)}\alpha^{-q(-m_d(1) + 2(k_d(1))}}  \cdots \frac{ Q^{k_d(d)r_d}\alpha^{-r_d(-m_d(d) + 2k_d(d))}}{(1- Q^{qk_d(d)}\alpha^{-q(-m_d(d) + 2k_d(d)})} \\
&+\sum_{k_d =0}^{m_d}\binom{m_d}{m_d - k_d}  (-1)^{m_d - k_d} \frac{ Q^{r_d (k_{d-1}(1) + m_d - k_d)}\alpha^{-r_d(-m_d(1) + 2k_{d-1}(1)+2( m_d-k_d ))}}{(1- Q^{q(k_{d-1}(1) + m_d - k_d)}\alpha^{-q(-m_d(1) + 2k_{d-1}(1)+2( m_d-k_d ))})} \times  \cdots \\
& \times    \frac{ Q^{(m_d- k_d)r_d}\alpha^{-r_d(m_d - 2k_d)}}{(1- Q^{q(m_d- k_d)}\alpha^{-q(m_d - 2k_d})}\Bigg].
 \end{split}
 \end{align}
 By continuing in this process for each index $k_t$ where $t = d-1, d-2, \ldots, 1$ and using the notations \eqref{eq5.55}, we have 
  \begin{align}\label{eq5.56}
  \begin{split}
 \zeta_{U}^{d}&(-{\bf m }\mid {\bf r})= \frac{D^{-m_d(1) /2}}{2^d} \sum_{k_1 =0}^{m_1}\binom{m_1}{k_1}  \times \cdots \times \sum_{k_d =0}^{m_d}\binom{m_d}{k_d} \\
 & \Bigg[\prod_{t =1}^{d}\theta_t(k_t,\ldots, k_d; m_t) + \sum_{j =1}^{d}\left(\prod_{t =1}^{d}\theta_t(k_t,\ldots, \overline{k_j},\ldots, k_d; m_t)\right)\\
 & + \sum_{1\leq i < j \leq d}\left(\prod_{t =1}^{d}\theta_t(k_t,\ldots, \overline{k_i},\ldots, \overline{k_j},\ldots, k_d; m_t)\right)+\cdots + \prod_{t =1}^{d}\theta_t(\overline{k_t},\ldots, \overline{k_d}; m_t)\Bigg].
  \end{split}
 \end{align}
Now rest of proof is similar to the proof of Theorem 7 in \cite{Meher2018}. Hence the result follows.
 \end{proof}
 
The following theorem deals with the special values of multiple Lucas  $L$-function for a quadratic character.
 \begin{theorem}\label{th15}
For a positive integer $d \geq 1,$ let ${\bf m} = (m_1, \ldots, m_d) \in \mathbb{N}^d$. For $P, Q \in \mathbb{Q}, \sqrt{D}$ an irrational number and $\chi_1, \ldots, \chi_d$ quadratic characters, $\mathcal{L}_U(\bf -{\bf m}\mid \chi)$ is rational valued except for the singularities.
 \end{theorem}
  \begin{proof}
Note that
 $$
 \mathcal{L}_U(\bf -{\bf m}\mid \chi)
=  \sum_{r_1 = 1}^{q}\chi_1(r_1) \sum_{r_2 = 1}^{q}\chi_{2}(r_1 + r_2)\cdots \sum_{r_d=1}^{q}\chi_d(r_1+\cdots + r_d)  \ \ \zeta_{U}^{d}(\bf  -{\bf m }\mid \bf r).
$$
Since $\chi $ is a quadratic character and by by Proposition \ref{prop14} $\zeta_{U}^{d}(\bf  -{\bf m }\mid \bf r)$ is rational except the singularities. This completes the proof. 
 \end{proof}
 
The following result can be proved in the same line as Theorem 7 in \cite{Meher2018} under some mild conditions.
 \begin{theorem}\label{th16}
For a positive integer $d \geq 1,$ let ${\bf m} = (m_1, \ldots, m_d) \in \mathbb{N}^d$. Let $P, Q \in \mathbb{Q}, \sqrt{D}$ be an irrational number and $f_1, \ldots, f_d$ are additive characters with $\prod_{k=i}^{d}\left|f_{k}(1)\right|\leq 1$ and $\prod_{k=i}^{d}\left|f_{k}(1)\right|$ is  rational valued for all $1\leq i\leq d$. Then $\mathcal{L}_{U}^{d}(\bf -m \mid \bf f)$ is rational except for singularities.
 \end{theorem}

\end{document}